\documentclass[12pt,reqno]{amsart}
\usepackage{amssymb,amsmath,amsthm,hyperref,mathrsfs,multirow,xcolor}
\usepackage{tikz}
\usetikzlibrary{arrows,matrix,positioning}
\newtheorem{theorem}{Theorem}
\newtheorem{lemma}[theorem]{Lemma}
\newtheorem{corollary}[theorem]{Corollary}

\newtheorem{proposition}[theorem]{Proposition}
\newtheorem*{thmNoNum}{Theorem}

\theoremstyle{definition}
\newtheorem{definition}[theorem]{Definition}
\newtheorem{example}[theorem]{Example}

\theoremstyle{remark}
\newtheorem*{remark}{Remark}

\allowdisplaybreaks

\begin{document}
%%BEGIN MACROS%%%%%%%%%%%%%%%%%%%%%%%%%%%%%%%%%%%%%%%%%%%%%%%%%%%%%%%%%%%%%%%
\newcommand\mylabel[1]{\label{#1}}
\newcommand{\beqs}{\begin{equation*}}
\newcommand{\eeqs}{\end{equation*}}
\newcommand{\beq}{\begin{equation}}
\newcommand{\eeq}{\end{equation}}
\newcommand\eqn[1]{(\ref{eq:#1})}
\newcommand\exam[1]{\ref{exam:#1}}
\newcommand\thm[1]{\ref{thm:#1}}
\newcommand\lem[1]{\ref{lem:#1}}
\newcommand\propo[1]{\ref{propo:#1}}
\newcommand\corol[1]{\ref{cor:#1}}
\newcommand\sect[1]{\ref{sec:#1}}
\newcommand\subsect[1]{\ref{subsec:#1}}

% definitions

\newcommand{\Z}{\mathbb Z}
\newcommand{\N}{\mathbb N}
\newcommand{\R}{\mathbb R}
\newcommand{\C}{\mathbb C}
\newcommand{\Q}{\mathbb Q}
\newcommand{\leg}[2]{\genfrac{(}{)}{}{}{#1}{#2}}
\newcommand{\bfrac}[2]{\genfrac{}{}{}{0}{#1}{#2}}
\newcommand{\sm}[4]{\left(\begin{smallmatrix}#1&#2\\ #3&#4 \end{smallmatrix} \right)}
\newcommand{\cPev}{{\mathcal P}^{\mathrm{ev}}}
\newcommand{\cD}{\mathcal D}
\newcommand{\cB}{\mathcal B}
\newcommand{\ev}{\mathrm{even}}
\newcommand{\od}{\mathrm{odd}}
\newcommand{\SL}{\mathrm{SL}}
\newcommand{\PGL}{\mathrm{PGL}}

\title{Period polynomial relations between formal double zeta values of odd weight
%\thanks{Grants or other notes
%about the article that should go on the front page should be
%placed here. General acknowledgments should be placed at the end of the article.}
}
%\subtitle{Do you have a subtitle?\\ If so, write it here}

%\titlerunning{Short form of title}        % if too long for running head

\author{Ding Ma}
\address{Department of Mathematics, University of Arizona, Tucson, AZ. 85721}
\email{martin@math.arizona.edu}

\date{\today}

%\authorrunning{Short form of author list} % if too long for running head
% The correct dates will be entered by the editor

\maketitle

\begin{abstract}
For odd $k$, we give a formula for the relations between double zeta values $\zeta(r,k-r)$ with $r$ even. This formula provides a connection with the space of cusp forms on $\SL_2(\Z)$. This is the odd weight analogue of a result in \cite{GKZ} by Gangl, Kaneko and Zagier. We also provide an answer of a question asked by Zagier in \cite{Zag1} about the left kernel of some matrix. Although the restricted sum statement in \cite{GKZ} fails in the odd weight case, we provide an asymptotical statement that replaces it. Our statement works more generally for restricted sums with any congruence condition on the first entry of the double zeta value.
%Insert your abstract here. Include keywords, PACS and mathematical
%subject classification numbers as needed.
\end{abstract}

%%SECTION 1%%%%%%%%%%%%%%%%%%%%%%%%%%%%%%%%%%%%%%%%%%%%%%%%%%%%%%%%%%%%%%%%%%%%%%%
\section{Introduction and main result}\label{introduction}

The double zeta values, which are defined for integers $r\geq 2$, $s\geq 1$ by
\begin{align}
\zeta(r,s)=\sum_{m>n>0}\frac{1}{m^r n^s},
\end{align}
satisfy numerous relations. The double shuffle relations on double zeta values are given by the following two sets of well-known relations (cf., e.g., \cite{Gon1}, \cite{IKZ}, \cite{Rac}):
\begin{align}
&\zeta(r,s)+\zeta(s,r)=\zeta(r)\zeta(s)-\zeta(k)\quad (r+s=k,r,s\geq 2), \label{eq:rel(1)} \\
&\sum_{r=2}^{k-1}\bigg[{r-1\choose j-1}+{r-1\choose k-j-1}\bigg]\zeta(r,k-r)=\zeta(j)\zeta(k-j)\quad (2\leq j\leq \frac{k}{2}). \label{eq:rel(2)}
\end{align}

Often people work in the formal double zeta space $\cD_k$ generated by formal symbols $Z_{r,s}$, $P_{r,s}$ and $Z_k$ satisfying the above two sets of relations, with $\zeta(r,s)$, $\zeta(r)\zeta(s)$ and $\zeta(k)$ replaced by $Z_{r,s}$, $P_{r,s}$ and $Z_k$, respectively. The advantage of this space is that we can work in it purely algebraically, since the double zeta values may satisfy other relations than those generated by \eqn{rel(1)} and \eqn{rel(2)}.

Many authors have studied the relations which can be deduced from the above two sets of relations. One of the most famous results in this area concerns the following sort of relations, which gives the first connection with modular forms:
\begin{thmNoNum}[Theorem 3 (Rough statement) in \cite{GKZ}] 
The values $\zeta(\od,\od)$ of weight $k$ satisfy at least $\dim S_k$ linearly independent relations, where $S_k$ denotes the space of cusp forms of weight $k$ on $\SL_2(\Z)$.
\end{thmNoNum}

\begin{example}
For $k=12$ and $k=16$, the first two cases for which there are non-zero cusp forms on $\SL_2(\Z)$, we have the following identities.
\begin{align*}
\frac{5197}{691}\zeta(12)&=28\zeta(9,3)+150\zeta(7,5)+168\zeta(5,7)\\
\frac{78967}{3617}\zeta(16)&=66\zeta(13,3)+375\zeta(11,5)+686\zeta(9,7)+675\zeta(7,9)+396\zeta(5,11).
\end{align*}
\end{example}

In their paper \cite{GKZ}, Gangl, Kaneko and Zagier proved the following general result for the formal double zeta space $\cD_k$ instead. We denote by $\cPev_k$ the subspace of $\cD_k$ spanned by the $P_{\ev,\ev}$. Let $W^-_k$ denote the space of even period polynomials of weight $k$ (see Section \ref{background}).

\begin{thmNoNum}[Theorem 3 in \cite{GKZ}] 
The spaces $\cPev_k$ and $W^-_k$ are canonically isomorphic to each other. More precisely, to each $p\in W^-_k$ we associate the coefficients $p_{r,s}$ and $q_{r,s}$ $(r+s=k)$ which are defined by $p(X,Y)=\sum {k-2 \choose r-1}p_{r,s}X^{r-1}Y^{s-1}$ and $p(X+Y,Y)=\sum {k-2 \choose r-1}q_{r,s}X^{r-1}Y^{s-1}$. Then $q_{r,s}-q_{s,r}=p_{r,s}$ (in particular $q_{r,s}=q_{s,r}$ for $r,s$ even) and 
\begin{align}
\sum_{\substack{r+s=k \\ r,s\textrm{ even}}}q_{r,s}Z_{r,s}\equiv 3\sum_{\substack{r+s=k \\ r,s\textrm{ odd}}}q_{r,s}Z_{r,s}\quad (\textrm{mod } Z_k),
\end{align}
and conversely, an element $\sum_{r,s\textrm{ odd}}c_{r,s}Z_{r,s}\in\cD_k$ belongs to $\cPev_k$ if and only if $c_{r,s}=q_{r,s}$ arising in this way.
\end{thmNoNum}
By taking the double zeta value realization, the above result for double zeta values follows directly. 

Although the above result is known for double zeta values of even weight, a direct connection between double zeta values of odd weight and the spaces of cusp forms was so far unknown. In \cite{Zag1}, Zagier proved the following result:
\begin{thmNoNum}[Theorem 3 in \cite{Zag1}] 
For each odd integer $k=2K+1\geq 5$, the numbers $\{\zeta(k-2r-1,2r+1)|\ r=0,\ldots,K-1\}$ satisfy $\dim S_{k-1}+\dim S_{k+1}$ relations, where $S_i$ denotes the space of cusp forms of weight $i$ on $\SL_2(\Z)$.
\end{thmNoNum}

The above result suggests a connection between the relations of odd weight double zeta values and the spaces of cusp forms. In this paper, we will prove the following results, which will establish the ``missing'' direct connection as in the even weight case. One of them provides relations from $S_{k-1}$, and the other one from $S_{k+1}$.

\begin{theorem}[Type I]\label{thm:1}
Let $k\geq 12$ be an even integer. To each $p\in W^+_k$ we associate the coefficients $b_{r,s}$ $(r+s=k+1)$ which are defined by 
$$p(X+Y,Y)=\sum_{r+s=k+1} {k-1 \choose r-1}b_{r,s}X^{r-1}Y^{s-2}.$$
Then 
\begin{align}\label{eq:rel1}
\sum_{\substack{r+s=k+1 \\ 4 \leq r\leq k-2 \textrm{ even}}}(b_{r,s}-b_{s,r})Z_{r,s}\equiv 0 \quad (\textrm{mod } Z_{k+1}).
\end{align}
\end{theorem}

\begin{theorem}[Type II]\label{thm:2}
Let $k\geq 12$ be an even integer. To each $p\in W^-_k$ we associate the coefficients $c_{r,s}$ $(r+s=k-1)$ which are defined by 
$$\frac{\partial}{\partial X}p(X+Y,Y)=\sum_{r+s=k-1} {k-3 \choose r-1}c_{r,s}X^{r-1}Y^{s-1}.$$

Then 
\begin{align}\label{eq:rel2}
\sum_{\substack{r+s=k-1 \\ 4 \leq r\leq k-4 \textrm{ even}}}(c_{r,s}-c_{s,r})Z_{r,s}\equiv 0 \quad (\textrm{mod } Z_{k-1}).
\end{align}
\end{theorem}

Using Theorems \ref{thm:1} and \ref{thm:2} we also obtain the following lower bounds for relations among double zeta values in odd weight, akin to the ``rough statement'' of Theorem 3 in \cite{GKZ}.

\begin{theorem}\label{thm:3}
Let $k\geq 7$ be an odd integer. Up to rational multiples of $\zeta(k)$, the values $\{\zeta(r,s)|\ r\textrm{ even, }4\leq r\leq k-3,\ r+s=k\}$ satisfy at least $\dim S_{k-1}+\dim S_{k+1}$ linearly independent rational linear relations, where $S_i$ denotes the space of cusp forms of weight $i$ on $\SL_2(\Z)$.
\end{theorem}

Notice that the first entry of the double zeta value in this result is an even integer between $4$ and $k-3$, while in Zagier's result it is an even integer between $2$ and $k-1$. It is worth pointing out that Theorems \ref{thm:1} and \ref{thm:2} are compatible with the decomposition (\ref{eq:bk}). The detail of such compatibility will be explained in Section \ref{Zagier}.

%explain the word 'at least'%

\begin{example}\label{exam:relation}
For $k=11$, $k=13$ and $k=15$, the only cases when $\dim S_{k-1}+\dim S_{k+1}=1$, we have\begin{align*}
-3\zeta(11)&=28\zeta(8,3)+20\zeta(6,5)-42\zeta(4,7);\\
-3\zeta(13)&=24\zeta(10,3)+28\zeta(8,5)-10\zeta(6,7)-36\zeta(4,9);\\
-3\zeta(15)&=22\zeta(12,3)+30\zeta(10,5)+7\zeta(8,7)-20\zeta(6,9)-33\zeta(4,11).
\end{align*}

For $k=17$, the first case when $\dim S_{k-1}=\dim S_{k+1}=1$, we have 
\begin{align*}
-23\zeta(17)=156&\zeta(14,3)+242\zeta(12,5)\\
+&153\zeta(10,7)-56\zeta(8,9)-215\zeta(6,11)-234\zeta(4,13);\\
-597\zeta(17)=4004&\zeta(14,3)+6358\zeta(12,5)\\
+&4347\zeta(10,7)-1624\zeta(8,9)-5885\zeta(6,11)-6006\zeta(4,13),
\end{align*}
where the first identity comes from $S_{16}$, and the second one from $S_{18}$.
\end{example}

Gangl, Kaneko and Zagier also proved the following statement in their paper \cite{GKZ}.
\begin{thmNoNum}[Theorem 1 in \cite{GKZ}] 
For even $k > 2$, one has
\begin{align*}
\sum^{k-1}_{\substack{r=2\\r\textrm{ even}}}Z_{r,k-r}=\frac{3}{4}Z_k,\qquad \sum^{k-1}_{\substack{r=2\\r\textrm{ odd}}}Z_{r,k-r}=\frac{1}{4}Z_k.
\end{align*}
\end{thmNoNum}

The double zeta value realization of the above statement tells us that for even $k>2$, we always have
\begin{align*}
\sum^{k-1}_{\substack{r=2\\r\textrm{ even}}}\zeta(r,k-r)=\frac{3}{4}\zeta(k),\qquad \sum^{k-1}_{\substack{r=2\\r\textrm{ odd}}}\zeta(r,k-r)=\frac{1}{4}\zeta(k).
\end{align*}

It is easy to see that the above statement does not hold for double zeta values of odd weight. But asymptotically, it is still correct, i.e., we have the following statement.

\begin{theorem}\label{thm:5}
For any integer $d\geq 2$, and any $i$ satisfying $0\leq i\leq d-1$, we have
\begin{align*}
\lim_{k\to\infty} \zeta(k)^{-1} {\displaystyle\sum^{k-1}_{\substack{r=2\\r\equiv i\bmod d}}}\zeta(r,k-r)=C_{d}^{(i)}:=\sum^{\infty}_{\substack{j=2\\j\equiv i\bmod d}}(\zeta(j)-1).
\end{align*}
\end{theorem}

\begin{remark}
Notice that we do not require our $k$ to be odd in this case. It is worth pointing out that our result above is compatible with Theorem $1.1$ in \cite{Mac} by Machide, which gives some restricted sum formulas for double zeta values. For example, when $k\equiv 0\bmod 3$, Machide proves that
\begin{align*}
\bigg(\sum_{r\equiv3(6)}-\sum_{r\equiv4(6)}-\sum_{r\equiv5(6)}\bigg)\zeta(r,k-r)=\frac{1}{3}\sum_{r\equiv1(2)}\zeta(r,k-3),
\end{align*}
while Theorem \ref{thm:5} gives that
\begin{align*}
\lim_{k\to\infty}\zeta(k)^{-1}\bigg(\sum_{r\equiv3(6)}-\sum_{r\equiv4(6)}-\sum_{r\equiv5(6)}\bigg)\zeta(r,k-r)
=\frac{1}{3}\lim_{k\to\infty}\zeta(k)^{-1}\sum_{r\equiv1(2)}\zeta(r,k-3).
\end{align*}
In this case, Machide's result works for all $k\equiv 0\bmod 3$, while Theorem \ref{thm:5} gives an asymptotical statement.
\end{remark}

\begin{example}
In particular, when $d=2$ and $i=0,1$, we have (cf., e.g., \cite{BBC})
\begin{align*}
C_{2}^{(0)}&=\sum_{n=1}^\infty(\zeta(2n)-1)=\frac{3}{4},\qquad C_{2}^{(1)}=\sum_{n=1}^\infty(\zeta(2n+1)-1)=\frac{1}{4}.
\end{align*}

Therefore, 
\begin{align*}
\zeta(k)^{-1}{\displaystyle\sum^{k-1}_{\substack{r=2\\r\textrm{ even}}}}\zeta(r,k-r)\to \frac{3}{4},\qquad \zeta(k)^{-1}{\displaystyle\sum^{k-1}_{\substack{r=2\\r\textrm{ odd}}}}\zeta(r,k-r)\to \frac{1}{4},\qquad\textrm{ as }k\to\infty.
\end{align*}
\end{example}

In Section $2$ we provide some background on the formal double zeta space and the $\PGL_2(\Z)$-action on the space of homogeneous polynomials. In Section $3$ we provide the proof of Theorem \thm{1}. Theorem \thm{2} can be proved using almost the same method, so we will only provide the construction and skip the detailed proof. Some double zeta value examples for Theorem \thm{1} and Theorem \thm{2} will be provided in Section $4$. In Section $5$, we will explain how to use our theorems to obtain information about the left kernel of Zagier's matrix $\cB_K$ (see Section 5). In Section $6$, we will show that all the rational linear relations obtained from Theorem \thm{1} and Theorem \thm{2} are linearly independent. Finally, in Section $7$, we will prove Theorem \thm{5} and provide some more examples of restricted sums of double zeta values.

%%%

%%SECTION 2%%%%%%%%%%%%%%%%%%%%%%%%%%%%%%%%%%%%%%%%%%%%%%%%%%%%%%%%%%%%%%%%%%%%%%%
\section{Background}\label{background}

We begin by reviewing the definition of the formal double zeta space (cf., \cite{Gon1}, \cite{IKZ}, \cite{Rac}). Let $k>2$ be an integer. We introduce formal variables $Z_{r,s}$, $P_{r,s}$ and $Z_k$ and impose the relations
\begin{align}
Z_{r,s}+Z_{s,r}=P_{r,s}-Z_k\quad(r+s=k), \label{eq:dzv1}\\
\sum_{r+s=k}\bigg[{r-1\choose i-1}+{r-1\choose j-1}\bigg]Z_{r,s}=P_{i,j}\quad(i+j=k). \label{eq:dzv2}
\end{align}
(From now on, whenever we write $r + s = k$ or $i + j = k$ without comment, it is assumed that the variables are integers $\geq 1$.) 

The formal double zeta space is defined as the $\Q$-vector space 
$$\cD_k=\frac{\{\textrm{$\Q$-linear combinations of formal symbols $Z_{r,s}$, $P_{r,s}$, $Z_k$}\}}{\langle \textrm{relations }\eqn{dzv1}\textrm{ and }\eqn{dzv2}\rangle}.$$

The double zeta realization we consider in this paper is the following realization $\cD_k\to\R$ of the formal double zeta space. 
\begin{align*}
Z_{r,s}&\mapsto 
\begin{cases}
\zeta(r,s), &\textrm{ if $r>1$,}\\
\kappa, &\textrm{ if $r=1$,}\\
\end{cases}\\
P_{r,s}&\mapsto
\begin{cases}
\zeta(r)\zeta(s), &\textrm{ if $r,s>1$,}\\
\kappa+\zeta(k-1,1)+\zeta(k), &\textrm{ if $r=1$ or $s=1$,}\\
\end{cases}\\
Z_k &\mapsto \zeta(k),
\end{align*}
where $\kappa\in\R$ can be chosen to be any real number.

One basic way of working with $\cD_k$ is by studying the relations among the $Z_{r,s}$. We first introduce some basic notation. For each even $k$, let $V_k=\langle X^{r-1}Y^{s-1}\ |\ r+s=k\rangle$ be the space of homogeneous polynomials of degree $k-2$ in two variables. Let $W_k\subset V_k$ be the subspace of polynomials satisfying the relations
\begin{align*}
P(X,Y)+P(-Y,X)&=0\\
P(X,Y)+P(X-Y,X)+P(Y,Y-X)&=0.
\end{align*}

We call $P\in W_k$ a period polynomial. This period polynomial space splits as the direct sum of subspaces $W_k^+$ and $W_k^-$ of polynomials which are symmetric and antisymmetric with respect to $X\leftrightarrow Y$. We call them odd and even period polynomials. The Eichler-Shimura-Manin theory tells us that there are canonical isomorphisms over $\C$ between $S_k$(the space of cusp forms of weight $k$) and $W_k^+$ and between $M_k$(the space of modular forms of weight $k$) and $W_k^-$.

In \cite{GKZ}, Gangl, Kaneko and Zagier proved the following statement, which is important in understanding the connection between relations of $Z_{r,s}$ up to $Z_k$ and the period polynomials.

\begin{proposition}[Proposition 2 in \cite{GKZ}] \label{propo:2}
Let $a_{r,s}$ and $\lambda$ be rational numbers. Then the following two statements are equivalent:
\begin{enumerate}
\item The relation
\begin{align}
\sum_{r+s=k}a_{r,s}Z_{r,s}=\lambda Z_k
\end{align}
holds in $\cD_k$.
\item The generating function
\begin{align}
A(X,Y)=\sum_{r+s=k}{k-2 \choose r-1}a_{r,s} X^{r-1} Y^{s-1}\in V_k
\end{align}
can be written as $H(X,X+Y)-H(X,Y)$ for some symmetric homogeneous polynomial $H\in \Q[X,Y]$ of degree $k-2$, and 
\begin{align}\label{eq:constant}
\lambda=\frac{k-1}{2}\int_0^1 H(t,1-t)dt.
\end{align}
\end{enumerate}
\end{proposition}

The last thing we want to review is the $\PGL_2(\Z)$-action on $V_k$ and an alternative definition of $W_k$ using this action. Let $F\in V_k$ be a homogeneous polynomial of degree $k-2$ in $X$ and $Y$, and $\gamma=(\begin{smallmatrix}a&b\\c&d\end{smallmatrix})\in \PGL_2(\Z)$. The $\PGL_2(\Z)$-action on $V_k$ is defined to be
$$(F|\gamma)(X,Y)=F(aX+bY,cX+dY).$$

There are $5$ important elements in $\PGL_2(\Z)$ which will be used later.
\begin{align*}
\varepsilon=
\begin{pmatrix}
0 & 1\\
1 & 0
\end{pmatrix},\ 
S=
\begin{pmatrix}
0 & -1\\
1 & 0
\end{pmatrix},\ 
U=\begin{pmatrix}
1 & -1\\
1 & 0
\end{pmatrix},\ 
T=US=
\begin{pmatrix}
1 & 1\\
0 & 1
\end{pmatrix},\ 
T'=U^2S=
\begin{pmatrix}
1 & 0\\
1 & 1
\end{pmatrix}.
\end{align*}

By using the above matrices, the space $W_k$ can also be defined as 
\begin{align}
W_k=\ker(1+S)\cap\ker(1+U+U^2)\subset V_k.
\end{align}

%%SECTION 3%%%%%%%%%%%%%%%%%%%%%%%%%%%%%%%%%%%%%%%%%%%%%%%%%%%%%%%%%%%%%%%%%%%%%%%
\section{Proof of Theorem 1}\label{proof1}
Having reviewed formal double zeta space, period polynomials and the $\PGL_2(\Z)$-action, we now turn our attention to proving our main theorem. We will only give the detailed proof for Theorem \thm{1} here. Theorem \thm{2} can be treated by the same method, so we only provide the corresponding construction in a remark.

\begin{proof}[Proof of Theorem 1]
Let $q=p| T$. Since $p$ is an odd period polynomial, it must be symmetric. We have $p(X+Y,X)=p(X,X+Y)$. Let $f=q\cdot Y-q|\varepsilon\cdot X$. First we want to show that $f=f|ST'$. By a direct computation, we have
\begin{align*}
f|ST'-f &= (q\cdot Y-q|\varepsilon\cdot X)|ST'-(q\cdot Y-q|\varepsilon\cdot X)\\
&= q|ST'\cdot X-q|\varepsilon ST'\cdot(-(X+Y))-(q\cdot Y-q|\varepsilon\cdot X)\\
&= (q|\varepsilon+q|ST'+q|\varepsilon ST')\cdot X+ (q|\varepsilon ST'-q)\cdot Y.
\end{align*}

I claim that the two terms in parentheses are both zero. 
\begin{align*}
q|\varepsilon+q|ST'+q|\varepsilon ST'&= p|T\varepsilon+p|TST'+p|T\varepsilon ST'\\
&= p(X+Y, X)+p(-Y,X)+p(-Y,-X-Y)\\
&= p(X, X+Y)-p(X,Y)+p(X+Y,Y)\\
&= 0;\\
q|\varepsilon ST'-q &= p|T\varepsilon ST'-p|T\\
&= p(-Y,-X-Y)-p(X+Y,Y)\\
&= p(X+Y,Y)-p(X+Y,Y)\\
&= 0.
\end{align*}

Hence we have shown that $f|ST'=f$. Now let us consider the function $f|S$. Since
\begin{align*}
(f|S)|\varepsilon-f|S &= (q\cdot Y-q|\varepsilon\cdot X)|S\varepsilon-(q\cdot Y-q|\varepsilon\cdot X)|S\\
&=(q|\varepsilon S\varepsilon-q|S)\cdot X+(q|S\varepsilon-q|\varepsilon S)\cdot Y\\
&= 0,
\end{align*}
we know that $f|S$ is a symmetric homogeneous polynomial of degree $k-1$. By applying Proposition \propo{2} to the following identity
$$f-f|S=f|ST'-f|S=(f|S)|(T'-1),$$
the coefficients of $f-f|S$ will give us a relation between $Z_{r,s}$ $(r+s=k+1)$ up to a scalar multiple of $Z_{k+1}$. The last thing we need to show is that the only nonzero terms of $X^{r-1}Y^{s-1}$ appearing in $f-f|S$ are $2{k-1\choose r-1}(b_{r,s}-b_{s,r})X^{r-1}Y^{s-1}$ for even $r$ satisfying $4\leq r\leq k-2$.

Since $f|S$ is symmetric, $f-f|S=f|SS-f|S$ only contains the terms with odd powers of $X$ between $3$ and $k-3$, 
%Done
%why no terms of odd powers of X of degree 1 and k-1 appear%
and those coefficients will be double of the corresponding ones in $f$. (There are no terms of odd powers of $X$ of degree $1$ and $k-1$ since $q$ itself already does not have such terms.) According to the definition, the coefficient of $X^{r-1}Y^{s-1}$ in $f=q\cdot Y-q|\varepsilon\cdot X$ is 
$${k-1\choose r-1}b_{r,s}-{k-1\choose s-1}b_{s,r}={k-1\choose r-1}(b_{r,s}-b_{s,r}).$$

Therefore, after dividing by $2$, we have shown the exact relation claimed in Theorem \thm{1}.
\end{proof}

\begin{remark}
For the proof of Theorem \thm{2}, we need to take $q=\frac{\partial}{\partial X}p| T$ and $f=q-q|\varepsilon$. Again we have
$$f=f|ST'\Longrightarrow f-f|S=f|ST'-f|S=(f|S)|(T'-1).$$
\end{remark}

\begin{corollary}
For type I and type II, we have the following two formulas.
\begin{itemize}
\item \textbf{Type I:} For any $p\in W_k^+$, let
\begin{eqnarray*}
L_1:=\frac{p(X+Y,Y)Y-p(X+Y,X)X-p(-X+Y,Y)Y-p(Y-X,-X)X}{2},
\end{eqnarray*}
then
\begin{align*}
L_1=\frac{f-f|S}{2},
\end{align*}
where $f(X,Y)=p(X+Y,X)Y-p(X+Y,Y)X$.

\item \textbf{Type II:}  For any $p\in W_k^-$, let $p'(X,Y)=\frac{\partial}{\partial X}p(X,Y)$, and
\begin{eqnarray*}
L_2:=\frac{p'(X+Y,Y)-p'(X+Y,X)-p'(-X+Y,Y)+p'(Y-X,-X)}{2}
\end{eqnarray*}
then 
\begin{align*}
L_2=\frac{f-f|S}{2},
\end{align*}
where $f(X,Y)=p'(X+Y,X)-p'(X+Y,Y)$.
\end{itemize}
Therefore, the coefficients of $Z_{r,s}$ in the relations (\ref{eq:rel1}) ((\ref{eq:rel2}) respectively)  are the coefficients of $X^{r-1}Y^{s-1}$ in $L_1$ ($L_2$ respectively) up to the obvious renormalization by the binomial coefficients.

\end{corollary}
\begin{remark}
Here ``the obvious renormalization by the binomial coefficients'' means that 
$$\begin{cases}
\textbf{Type I:} &\textrm{dividing the coefficient of $X^{r-1}Y^{s-1}$ in $L_1$ by ${k-1\choose r-1}$}\\
\textbf{Type II:} &\textrm{dividing the coefficient of $X^{r-1}Y^{s-1}$ in $L_2$ by ${k-3\choose r-1}$}
\end{cases},$$
\end{remark}

\begin{proof}[Proof of Corollary 1]
According to the proof of Theorem \thm{1}, we know that up to the obvious renormalization by the binomial coefficients, the linear relations can be computed by $\frac{f-f|S}{2}$, where $f(X,Y)=p(X+Y,X)Y-p(X+Y,Y)X$. It can be seen from the definition that
\begin{eqnarray*}
L_1=\frac{f-f|S}{2}.
\end{eqnarray*}
Similarly, for type II, the linear relations can be computed by $\frac{f-f|S}{2}$, where $f(X,Y)=p'(X+Y,X)-p'(X+Y,Y)$. Again we have 
\begin{eqnarray*}
L_2=\frac{f-f|S}{2}.
\end{eqnarray*}
\end{proof}

%Done
%Also, I think it is better to supply a brief account of the proof of Theorem 2. Though it may be parallel to Theorem 1, an indication how to choose corresponding q and f and how the formula corresponding to f-f|S=(f|S)|(T'-1) looks like, would be helpful to the readers.

%\begin{proof}[Proof of Theorem 3]
%First of all, we can get $\dim S_{k-1}$ relations from Theorem 1 and $\dim S_{k+1}$ relations relations from Theorem 2. Eichler-Shimura-Manin theory tells us that those $\dim S_{k-1}$ relations obtained from Theorem $1$ are linearly independent, since the corresponding period polynomials are linearly independent. The same fact holds for those $\dim S_{k+1}$ relations obtained from Theorem $2$. Therefore, all we need to show is that there can not be any linear relation between those two sets of relations.
%Notice that those relations obtained from Theorem $1$ are coming from odd period polynomial directly 
%\end{proof}

%%SECTION 4%%%%%%%%%%%%%%%%%%%%%%%%%%%%%%%%%%%%%%%%%%%%%%%%%%%%%%%%%%%%%%%%%%%%%%%
\section{Examples}\label{example}
Now we will provide some examples for Theorem \thm{1} and Theorem \thm{2}. By taking the double zeta value realization, we obtain the examples for double zeta values given in Example \exam{relation}
\begin{example}\label{exam:wt13}
The space $W^+_{12}$ is $1$-dimensional, spanned by the odd period polynomial $p(x)=4x^9-25x^7+42x^5-25x^3+4x$. We have $p(x+1)=4x^9+36x^8+119x^7+161x^6+21x^5-161x^4-144x^3-36x^2$, so the $b_{r,s}$ of the theorem are given (after multiplication by $330$) by the table
\begin{center}
\begin{tabular}{|c|c|c|c|c|c|c|c|c|}
\hline
$r$  & 10 & 9 & 8 & 7 & 6 & 5 & 4 & 3  \\ 
$s$ & 3 & 4 & 5 & 6 & 7 & 8 & 9 & 10 \\ \hline
$330b_{r,s}$ & 24 & 72 & 119 & 115 & 15 & -161 & -288 & -216 \\ \hline
\end{tabular}
\end{center}

The relations in the theorem, divided by $10$, become
$$24Z_{10,3}+28Z_{8,5}-10Z_{6,7}-36Z_{4,9}\equiv 0\quad \textrm{mod }Z_{13}.$$
In the double zeta value realization, this can be written as 
\begin{align}
24\zeta(10,3)+28\zeta(8,5)-10\zeta(6,7)-36\zeta(4,9)=-3\zeta(13).
\end{align}
The coefficient of $\zeta(13)$ can be obtained from \eqn{constant} directly.
\end{example}

\begin{example}\label{exam:wt11}
The space $W^-_{12}$ is $2$-dimensional, spanned by the even period polynomial $p(x)=x^{10}-1$ and $x^8-3x^6+3x^4-x^2$. For the latter, we have $p'(x+1)=8x^7+56x^6+150x^5+190x^4+112x^3+24x^2$, so the $c_{r,s}$ of the theorem are given (after multiplication by $63$) by the table
\begin{center}
\begin{tabular}{|c|c|c|c|c|c|c|}
\hline
$r$ & 8 & 7 & 6 & 5 & 4 & 3  \\ 
$s$ & 3  & 4 & 5 & 6 & 7 & 8 \\ \hline
$63c_{r,s}$ & 14 & 42 & 75 & 95 & 84 & 42 \\ \hline
\end{tabular}
\end{center}

The relations in the theorem, divided by $-1$, become
$$28Z_{8,3}+20Z_{6,5}-42Z_{4,7}\equiv 0\quad \textrm{mod }Z_{11}.$$
In the double zeta value realization, this can be written as 
\begin{align}\label{eq:wt11}
28\zeta(8,3)+20\zeta(6,5)-42\zeta(4,7)=-3\zeta(11).
\end{align}
\end{example}

\begin{example}\label{exam:wt15}
The space $W^-_{16}$ is $2$-dimensional, spanned by the even period polynomial $p(x)=x^{14}-1$ and $2x^{12}-7x^{10}+11x^8-11x^6+7x^4-2x^2$. For the latter, we have $p'(x+1)=24x^{11}+264 x^{10}+1250 x^9+3330 x^8+5488 x^7+5824 x^6+4050 x^5+1850 x^4 +528 x^3+72 x^2$, so the $c_{r,s}$ of the theorem are given (after multiplication by $\frac{429}{2}$) by the table
\begin{center}
\begin{tabular}{|c|c|c|c|c|c|c|c|c|c|c|}
\hline
$r$ & 12 & 11 & 10 & 9 & 8 & 7 & 6 & 5 & 4 & 3  \\ 
$s$ & 3  & 4 & 5 & 6 & 7 & 8 & 9 & 10 & 11 & 12\\ \hline
$\frac{429}{2}c_{r,s}$ & 66 &198 & 375 & 555 & 686 & 728 & 675 & 555 & 396 &198 \\ \hline
\end{tabular}
\end{center}

The relations in the theorem, divided by $-6$, become
$$22Z_{12,3}+30Z_{10,5}+7Z_{8,7}-20Z_{6,9}-33Z_{4,11}\equiv 0\quad \textrm{mod }Z_{15}.$$
In the double zeta value realization, this can be written as 
\begin{align}
22\zeta(12,3)+30\zeta(10,5)+7\zeta(8,7)-20\zeta(6,9)-33\zeta(4,11)=-3\zeta(15).
\end{align}
\end{example}

\begin{remark}
The reason why we do not consider the even period polynomials $x^{10}-1$ and $x^{14}-1$ in Example \exam{wt11} and Example \exam{wt15} is that they will always give us the trivial relation.
\end{remark}

%%SECTION 5%%%%%%%%%%%%%%%%%%%%%%%%%%%%%%%%%%%%%%%%%%%%%%%%%%%%%%%%%%%%%%%%%%%%%%%
\section{Relation to Zagier's problem}\label{Zagier}

Let us see how our relations \eqn{rel1} and \eqn{rel2} are related to Zagier's matrix $\cB_K$. Let $k=2K+1$. In \cite{Zag1}, Zagier obtained the following relation, which is first predicted by Euler without giving an explicit formula in \cite{Euler}.
\begin{align}\label{eq:bk}
&\zeta(k-2m-1,2m+1) \nonumber \\ 
&=\sum_{n=1}^K \bigg[\delta_{n,m}+\delta_{n,K}-{2n\choose 2m}-{2n\choose 2K-2m-1}\bigg]\zeta(2n+1)\zeta(k-2n-1),
\end{align}
where $\zeta(0)=-\frac{1}{2}$ by convention.

Let $\cB_K$ be the matrix coming from the above relations \eqn{bk}. For example, when $k=11$, the above relations can be written as
\begin{align*}
\begin{pmatrix}
\zeta(2,9)\\
\zeta(4,7)\\
\zeta(6,5)\\
\zeta(8,3)\\
\zeta(10,1)
\end{pmatrix}=
\begin{pmatrix}
    	-2 &   -4   & -6 &   -8&27\\
 	0  &  -4  & -20  & -84 &   \frac{329}{2}   \\
  	 0   &  0  & -21 & -126&  \frac{461}{2} \\
        0   & -6  & -15  & -36 & 82\\
      -1 &   -1 &   -1   & -1&5
\end{pmatrix}
\begin{pmatrix}
\zeta(8)\zeta(3)\\
\zeta(6)\zeta(5)\\
\zeta(4)\zeta(7)\\
\zeta(2)\zeta(9)\\
\zeta(11)
\end{pmatrix}=\cB_5
\begin{pmatrix}
\zeta(8)\zeta(3)\\
\zeta(6)\zeta(5)\\
\zeta(4)\zeta(7)\\
\zeta(2)\zeta(9)\\
\zeta(11)
\end{pmatrix}.
\end{align*}

Let us consider the following submatrix $\cB^{(1)}_5$ of $\cB_5$.
\begin{center}
\begin{tikzpicture}
        \matrix [matrix of math nodes,left delimiter=(,right delimiter=)] (m)
        {
    	-2 &   -4   & -6 &   \ -8&27\\
 	\ \ \ 0  &  -4  & -20  & -84 &   \frac{329}{2}   \\
   	\ \ \ 0   &  \ 0  & -21 & -126&  \frac{461}{2} \\
        	\ \ \ 0   & -6  & -15  & -36 & 82\\
     	 -1 &   -1 &   -1   & -1&5\\
        };  
        \draw[thick,dashed] (m-1-1.north west) -- (m-1-4.north east) -- (m-4-4.south east) -- (m-4-1.south west) -- (m-1-1.north west);
    \end{tikzpicture}
\end{center}

Since this submatrix $\cB_5^{(1)}$ corresponds to the linear expressions of $\{\zeta(8,3), \zeta(6,5),\zeta(4,7),\zeta(2,9)\}$ in terms of $\{\zeta(8)\zeta(3),\zeta(6)\zeta(5),\zeta(4)\zeta(7),\zeta(2)\zeta(9)\}$ up to scalar multiples of $\zeta(11)$, the relation \eqn{wt11} can be translated into the fact that the vector $(0,-42,20,28)$ lies in the left kernel of $\cB^{(1)}_5$. In general, the above argument proves the following statement.
\begin{proposition}
Let $k=2K+1$ be an odd integer. Let $\cB_K^{(1)}$ be the $(K-1)\times(K-1)$-minor of $\cB_K$ obtained by deleting the last columns and the last row of $\cB_K$. Then the vectors obtained from the coefficients of $Z_{r,s}$ in the linear relations \eqn{rel1} and \eqn{rel2} belong to the left kernel of $\cB_K^{(1)}$. 
\end{proposition}

Now let us see how to use our relation \eqn{wt11} to get Zagier's relation
\begin{align}\label{eq:wt11Zag}
-6\zeta(10,1)+17\zeta(8,3)+13\zeta(6,5)-27\zeta(4,7)+3\zeta(2,9)=0.
\end{align}
Or more generally for any weight $k=2K+1\geq 11$, let us see how to use our relations from Theorem \thm{1} and Theorem \thm{2} to get nontrivial elements in the left kernel of $\cB_K$ (i.e., kernel of $\cB^t_K$).

We have already known from \eqn{bk} the expression of $\zeta(r,s)$ in terms of linear combination of products of single zeta values when $r$ is even. By a direct computation, we can see that the following definition of canonical relations indeed give us relations between double zeta values of odd weight.

\begin{definition}
For any odd integer $k\geq 5$, we call the following relation \textit{the canonical relation in weight $k$}.
\begin{align}\label{eq:canon}
2(k-2)\zeta(k-1,1)+\sum_{\substack{r+s=k\\r\textrm{ even}}}(r-s)\zeta(r,s)-(k-2)\zeta(2,k-2)=-\frac{3}{4}(k-3)\zeta(k)
\end{align}
\end{definition}

\begin{example}
The first few canonical relations in lower weights are listed below.
\begin{align*}
-\frac{3}{2}\zeta(5)&=6\zeta(4,1)-3\zeta(2,3);\\
-3\zeta(7)&=10\zeta(6,1)-\zeta(4,3)-5\zeta(2,5);\\
-\frac{9}{2}\zeta(9)&=14\zeta(8,1)+3\zeta(6,3)-\zeta(4,5)-7\zeta(2,7);\\
-6\zeta(11)&=18\zeta(10,1)+5\zeta(8,3)+\zeta(6,5)-3\zeta(4,7)-9\zeta(2,9);\\
-\frac{15}{2}\zeta(13)&=22\zeta(12,1)+7\zeta(10,3)+3\zeta(8,5)-\zeta(6,7)-5\zeta(4,9)-11\zeta(2,11);\\
-9\zeta(15)&=26\zeta(14,1)+9\zeta(12,3)+5\zeta(10,5)+\zeta(8,7)-3\zeta(6,9)-7\zeta(4,11)-13\zeta(2,13).
\end{align*}
\end{example}

In particular, in weight $11$, we have both the canonical relation in weight $11$ and our relation \eqn{wt11}
\begin{align*}
-6\zeta(11)&=18\zeta(10,1)+5\zeta(8,3)+\zeta(6,5)-3\zeta(4,7)-9\zeta(2,9);\\
-3\zeta(11)&=28\zeta(8,3)+20\zeta(6,5)-42\zeta(4,7).
\end{align*}

Now we can easily see that by cancelling $\zeta(11)$ from the above two relations, we get exactly Zagier's relation \eqn{wt11Zag}
\begin{align*}
-6\zeta(10,1)+17\zeta(8,3)+13\zeta(6,5)-27\zeta(4,7)+3\zeta(2,9)=0.
\end{align*}

In general, for any odd integer $k=2K+1\geq 11$, by cancelling $\zeta(k)$ from both canonical relation of weight $k$ and some weight $k$ relation obtained from Theorem \thm{1} or Theorem \thm{2}, we will get a nontrivial element in the left kernel of $\cB_K$.

%p.8, section 5:  the connection to Zagier's B_k is only described by an example.  
%Have you proved the connection in general?  If so, you better write the 
%general statement and proof.

%%SECTION 4%%%%%%%%%%%%%%%%%%%%%%%%%%%%%%%%%%%%%%%%%%%%%%%%%%%%%%%%%%%%%%%%%%%%%%%
\section{Proof of Theorem 3}\label{proof3}

In this section, we will prove that the linear relations obtained from Theorem \thm{1} and Theorem \thm{2} are all linearly independent.

The proof is based on the following injective map proved by Zagier (c.f. \cite{Zag1} (41)).
\begin{eqnarray}
W_{2K}^+\oplus W_{2K+2}^-\to \ker(\cB_K).
\end{eqnarray}
According to the definition of the above map, the image of $W_{2K}^+\oplus W_{2K+2}^-$ always has zero as the last entry. Hence it naturally defines an injective map
\begin{eqnarray}\label{eq:image}
W_{2K}^+\oplus W_{2K+2}^-\to \ker(\cB_K^{(1)}).
\end{eqnarray}

We need the following lemma to relate the linear relations with the image of the above map.

\begin{lemma}
$L_1$ and $L_2$ in Corollary 1 can also be computed as follows.
\begin{eqnarray}
L_1&=&L_1':=\frac{p(Y,X+Y)(X+Y)-p(Y,-X+Y)(-X+Y)}{2}\\
L_2&=&L_2':=\frac{p'(Y,X+Y)-p'(Y,-X+Y)}{2}
\end{eqnarray}
\end{lemma}
\begin{proof}[Proof of Lemma 1]
\begin{itemize}
\item \textbf{Type I: }For any $p\in W^+_k$, we have $p(X,Y)=p(Y,X)$, therefore $p(X+Y,Y)Y=p(Y,X+Y)Y$ and $p(-X+Y,Y)Y=p(Y,-X+Y)Y$. Then
\begin{eqnarray*}
&&2(L_1'-L_1)\\
&=&X\bigg(\big(p(Y,X+Y)+p(X+Y,X)\big)+\big(p(Y,-X+Y)+p(Y-X,-X)\big)\bigg)\\
&=&X(p(X,Y)+p(-X,Y))\\
&=&X(p(X,Y)-p(X,Y))\\
&=&0.
\end{eqnarray*}
\item \textbf{Type II: }For any $p\in W^-_k$, we have 
$$p(X,X+Y)+p(X+Y,Y)+p(Y,X)=0.$$
Taking partial derivatives with respect to $Y$ term by term, we have
\begin{eqnarray}
\frac{\partial}{\partial Y}p(X,X+Y)&=&-\frac{\partial}{\partial Y}p(Y+X,X)=-p'(Y+X,X),\\
\frac{\partial}{\partial Y}p(X+Y,Y)&=&p'(X+Y,Y)-p'(Y,X+Y),\\
\frac{\partial}{\partial Y}p(Y,X)&=&p'(Y,X).
\end{eqnarray}
Summing over all the three terms above, we get
\begin{eqnarray}
p'(Y,X+Y)-p'(X+Y,Y)+p'(X+Y,X)=p'(Y,X).
\end{eqnarray}
Therefore,
\begin{eqnarray*}
&&2(L_2'-L_2)\\
&=&\bigg(p'(Y,X+Y)-p'(X+Y,Y)+p'(X+Y,X)\bigg)\\
&&-\bigg(p'(Y,-X+Y)-p'(-X+Y,Y)+p'(-X+Y,-X)\bigg)\\
&=&p'(Y,X)-p'(Y,-X)\\
&=&0.
\end{eqnarray*}
\end{itemize}
Hence we have proven the lemma.
\end{proof}

Now we are ready to prove the linear independence of the rational linear relations obtained from Theorem \thm{1} and Theorem \thm{2}.
\begin{proof}[Proof of Theorem 3]
First, we will relate those rational linear relations to the images (\ref{eq:image}) computed by Zagier.
\begin{itemize}
\item \textbf{Type I: }
For any $p\in W^+_k$, let us assume that $p(X,Y)=\sum_{r\textrm{ odd}} \alpha_{r,s}X^rY^s$. Notice that in this case, we have $r+s=k-2$.
\begin{eqnarray*}
L_1'&=&\frac{p(Y,X+Y)(X+Y)-p(Y,-X+Y)(-X+Y)}{2}\\
&=&\frac{1}{2}\bigg(\sum_{r\textrm{ odd}} \alpha_{r,s}Y^r(X+Y)^{s+1}-\sum_{r\textrm{ odd}} \alpha_{r,s}Y^r(-X+Y)^{s+1}\bigg)\\
&=&\sum_{r\textrm{ odd}}\sum_{i\textrm{ odd}} \alpha_{r,s}{s+1\choose i}X^iY^{k-1-i}\\
\end{eqnarray*}
Let us define two $\frac{k-2}{2}\times \frac{k-2}{2}$ matrices $D_1^{(k)}$ and $B_1^{(k)}$ by 
\begin{eqnarray}
(D_1^{(k)})^{-1}=\mathrm{diag}\bigg({k-1\choose 2i-1}\bigg),\qquad (B_1^{(k)})_{ij}={2j\choose 2i-1}.
\end{eqnarray}
By the above computation, we can see that left multiplication by $B_1^{(k)}$ of $(\alpha_{r,s})^T$ gives us a renormalization by  a factor of ${s+1\choose i}$ and further left multiplication by $D_1^{(k)}$ gives us the obvious renormalization by binomial coefficients. Therefore, the column vectors $D_1^{(k)}B_1^{(k)}(\alpha_{r,s})^T$ gives us the coefficients of the rational linear relations from Theorem \thm{1}.

\item \textbf{Type II: }
For any $p\in W^-_k$, let us assume that $p'(X,Y)=\sum_{r\textrm{ odd}} \beta_{r,s}X^rY^s$. Notice that in this case, we have $r+s=k-3$.
\begin{eqnarray*}
L_2'&=&\frac{p'(Y,X+Y)-p'(Y,-X+Y)}{2}\\
&=&\frac{1}{2}\bigg(\sum_{r\textrm{ odd}} \beta_{r,s}Y^r(X+Y)^{s}-\sum_{r,s\textrm{ odd}} \beta_{r,s}Y^r(-X+Y)^{s}\bigg)\\
&=&\sum_{r\textrm{ odd}}\sum_{i\textrm{ odd}} \beta_{r,s}{s\choose i}X^iY^{k-3-i}\\
\end{eqnarray*}
Let us define two $\frac{k-4}{2}\times \frac{k-4}{2}$ matrices $D_2^{(k)}$ and $B_2^{(k)}$ by 
\begin{eqnarray}
(D_2^{(k)})^{-1}=\mathrm{diag}\bigg({k-3\choose 2i-1}\bigg),\qquad (B_2^{(k)})_{ij}={2j\choose 2i-1}.
\end{eqnarray}
Similarly, we can see that left multiplication by $B_2^{(k)}$ of $(\beta_{r,s})^T$ gives us a renormalization by a factor of ${s\choose i}$ and further left multiplication by $D_2^{(k)}$ gives us the obvious renormalization by binomial coefficients. Therefore, the column vectors $D_2^{(k)}B_2^{(k)}(\beta_{r,s})^T$ gives us the coefficients of the rational linear relations from Theorem \thm{2}.
\end{itemize}
For a fixed odd weight $N$, according to the definition of $D_1^{(k)},D_2^{(k)},B_1^{(k)},B_2^{(k)}$, we have
\begin{eqnarray}
D_1^{(N-1)}=D_2^{(N+1)},\qquad B_1^{(N-1)}=B_2^{(N+1)}.
\end{eqnarray}
Moreover, $D_1^{(k)},D_2^{(k)}$ are always invertible diagonal matrices, and $B_1^{(k)},B_2^{(k)}$ are always invertible upper triangular matrices. The injectivity of (\ref{eq:image}) (i.e. the linear independence of $(\alpha_{r,s})^T$'s and $(\beta_{r,s})^T$'s) implies that for a fixed odd weight $N$, all the rational linear relations from Theorem \thm{1} and Theorem \thm{2} are linearly independent. Therefore, we have proven Theorem \thm{3}.
\end{proof}

\begin{remark}
Notice that the matrices $D_1^{(k)},D_2^{(k)},B_1^{(k)},B_2^{(k)}$ are similar to the ones defined in \cite{BS}.
\end{remark}

\begin{example}
Here we list two examples for the case when the weight is $11$ and $13$.
\begin{itemize}
\item $N=11$:

In this case, we only have the following linear relation of type II coming from the $W_{12}^-$.
\begin{eqnarray*}
3\zeta(11)&=0\zeta(2,9)+42\zeta(4,7)-20\zeta(6,5)-28\zeta(8,3).
\end{eqnarray*}
In this case,
$$
D_2^{(12)}=\left(\begin{array}{rrrr}
\frac{1}{9} & 0 & 0 & 0 \\
0 & \frac{1}{84} & 0 & 0 \\
0 & 0 & \frac{1}{126} & 0 \\
0 & 0 & 0 & \frac{1}{36}
\end{array}\right),\quad
B_2^{(12)}=\left(\begin{array}{rrrr}
2 & 4 & 6 & 8 \\
0 & 4 & 20 & 56 \\
0 & 0 & 6 & 56 \\
0 & 0 & 0 & 8
\end{array}\right),
$$
and we have
$$D_2^{(12)}B_2^{(12)}(4,-9,6,-1)^T=\bigg(0, \frac{1}{3}, -\frac{10}{63}, -\frac{2}{9}\bigg)^T=\frac{1}{126}(0,42,-20,-28)^T.$$

\item $N=13$:

In this case, we only have the following linear relation of type I coming from the $W_{12}^+$.
\begin{eqnarray*}
-3\zeta(13)&=0\zeta(2,11)-36\zeta(4,9)-10\zeta(6,7)+28\zeta(8,5)+24\zeta(10,3).
\end{eqnarray*}
In this case,
$$
D_1^{(12)}=\left(\begin{array}{rrrrr}
\frac{1}{11} & 0 & 0 & 0 & 0 \\
0 & \frac{1}{165} & 0 & 0 & 0 \\
0 & 0 & \frac{1}{462} & 0 & 0 \\
0 & 0 & 0 & \frac{1}{330} & 0 \\
0 & 0 & 0 & 0 & \frac{1}{55}
\end{array}\right),\quad
B_1^{(12)}=\left(\begin{array}{rrrrr}
2 & 4 & 6 & 8 & 10 \\
0 & 4 & 20 & 56 & 120 \\
0 & 0 & 6 & 56 & 252 \\
0 & 0 & 0 & 8 & 120 \\
0 & 0 & 0 & 0 & 10
\end{array}\right),
$$
and we have
$$D_1^{(12)}B_1^{(12)}(4,-25,42,-25,4)^T=\bigg(0, -\frac{12}{11}, -\frac{10}{33}, \frac{28}{33}, \frac{8}{11}\bigg)^T=\frac{1}{33}(0,-36,-10,28,24)^T.$$

\end{itemize}

\end{example}

%%SECTION 5%%%%%%%%%%%%%%%%%%%%%%%%%%%%%%%%%%%%%%%%%%%%%%%%%%%%%%%%%%%%%%%%%%%%%%%
\section{Proof of Theorem \thm{5}}\label{proof5}
In this section, we will give the proof of Theorem \thm{5}. The examples about all restricted sum with $d=3,$ $4$ will be given at the end of the section.
\begin{proof}[Proof of Theorem \thm{5}]
When $d=1$, the result is immediate from the following identities
\begin{align*}
\sum^{k-1}_{r=2}\zeta(r,k-r)=\zeta(k),\qquad \sum_{n=2}^\infty(\zeta(n)-1)=1.
\end{align*}

Without loss of generality, we may assume that $d\geq 2$. Since $\lim_{k\to\infty}\zeta(k)=1$, we only need to show the equality
\begin{align*}
\lim_{k\to\infty} \sum^{k-1}_{\substack{r=2\\r\equiv i\bmod d}}\zeta(r,k-r)=C_{d}^{(i)}:=\sum^{\infty}_{\substack{j=2\\j\equiv i\bmod d}}(\zeta(j)-1).
\end{align*}

We can rewrite this limit as
\begin{align}\label{eq:thm5}
\lim_{k\to\infty} \sum^{k-1}_{\substack{r=2 \\ r\equiv i\bmod d}}\bigg(\zeta(r,k-r)-(\zeta(r)-1)\bigg)=0.
\end{align}

Let $s=k-r$. By the definition of double zeta values, we have
\begin{align*}
\zeta(r,s)-(\zeta(r)-1)&=\frac{1}{3^r}\bigg(\frac{1}{2^s}\bigg)+\frac{1}{4^r}\bigg(\frac{1}{2^s}+\frac{1}{3^s}\bigg)+\cdots+\frac{1}{m^r}\bigg(\frac{1}{2^s}+\cdots+\frac{1}{(m-1)^s}\bigg)+\cdots\\
&< \frac{1}{3^r}(\zeta(s)-1)+\frac{1}{4^r}(\zeta(s)-1)+\cdots+\frac{1}{m^r}(\zeta(s)-1)+\cdots\\
&< (\zeta(r)-1)(\zeta(s)-1).
\end{align*}

Let us assume that $i+Nd\leq k<i+(N+1)d$ for some $N$. Taking the sum over all qualifying $r$ lying between $2$ and $k$, we get
\begin{align*}
\sum^{k-1}_{\substack{r=2 \\ r\equiv i\bmod d}}\bigg(\zeta(r,k-r)-(\zeta(r)-1)\bigg)&<\sum^{N}_{\substack{j=0 \\ 2\leq i+dj\leq k-2}} \bigg(\zeta(i+dj)-1\bigg)\bigg(\zeta(k-i-dj)-1\bigg)\\
&\leq (\zeta(2)-1)\cdot(N+1)\cdot \bigg(\zeta\bigg(\frac{k}{2}\bigg)-1\bigg)\\
&\leq (\zeta(2)-1)\cdot(N+1)\cdot (\zeta(N)-1)
\end{align*}

Since the right hand side goes to zero as $N$ goes to infinity by the following comparison
\begin{align*}
\lim_{N\to\infty}(N+1)(\zeta(N)-1)&=\lim_{N\to\infty}N(\zeta(N)-1)\\
&=\lim_{N\to\infty}N\bigg(\frac{1}{2^N}+\frac{1}{3^N}+\frac{1}{4^N}+\cdots\bigg)\\
&\leq \lim_{N\to\infty}N\bigg(\frac{1}{2^N}+\frac{1}{2^N}+\frac{1}{4^N}+\frac{1}{4^N}+\frac{1}{4^N}+\frac{1}{4^N}+\frac{1}{8^N}+\cdots\bigg)\\
&= \lim_{N\to\infty}N\bigg(\frac{1}{2^{N-1}}+\frac{1}{4^{N-1}}+\frac{1}{8^{N-1}}+\cdots\bigg)\\
&=\lim_{N\to\infty}\frac{N}{2^{N-1}-1}\\
&=0,
\end{align*}
we have shown \eqn{thm5}, i.e., we have proven Theorem \thm{5}.
\end{proof}

\begin{example}
For $d=1$, $2$, $3$, and $4$, we have the following computations.

\begin{center}
\begin{tabular}{|c|c|c|c|c|}
\hline
${\displaystyle\lim_{k\to\infty}\zeta(k)^{-1}{\displaystyle \sum^{k-1}_{\substack{r=1\\ r\equiv 0\bmod d}}}\zeta(r,k-r)}$ & $i=0$ & $i=1$ & $i=2$ & $i=3$\\ \hline
$d=1$ & 1 & & & \\ \hline
$d=2$ & $0.75$ & $0.25$ & & \\ \hline
$d=3$ & $0.22168939\ldots$ & $0.09180726\ldots$ & $0.68650334\ldots$ & \\ \hline
$d=4$ & $0.08666297\ldots$ & $0.03906700\ldots$ & $0.66333702\ldots$ & $0.21093299\ldots$ \\ \hline
\end{tabular}
\end{center}
\end{example}

\section*{Acknowledgement}
This study was funded by NSF Grant No. DMS-1401122. The author would like to thank Romyar Sharifi for introducing him a project related to this area, David Broadhurst for his numerical data, Masanobu Kaneko and Herbert Gangl for very useful comments on the first draft.

%\begin{acknowledgements}
%If you'd like to thank anyone, place your comments here
%and remove the percent signs.
%\end{acknowledgements}

% BibTeX users please use one of
%\bibliographystyle{spbasic}      % basic style, author-year citations
%\bibliographystyle{spmpsci}      % mathematics and physical sciences
%\bibliographystyle{spphys}       % APS-like style for physics
%\bibliography{}   % name your BibTeX data base

\begin{thebibliography}{10}

\bibitem{BS}
Baumard, S., Schneps, L.
\textit{Period polynomial relations between double zeta values}
The Ramanujan Journal, Volume 32, Issue 1, 83-100, 2013
%**********************************************************

\bibitem{BBC}
Borwein, J., Bradley, D., and Crandall, R.
\textit{Computational strategies for the Riemann zeta function},
Journal of Computational and Applied Mathematics 121, 247-296, 2000
%**********************************************************

\bibitem{Euler}
Euler, L.
\textit{Meditationes circa singulare serierum genus},
Novi Comm. Acad. Sci. Petropol. 20 (1775), 140-186, in Opera Omnia Ser. I, vol. 15, Teubner, Berlin, 217-267, 1927
%**********************************************************

\bibitem{GKZ}
Gangl, H., Kaneko, M., and Zagier, D.
\textit{Double zeta values and modular forms},
In Automorphic Forms and Zeta Functions, Proceedings of the Conference in Memory of Tsuneo Arakawa, S. B\:{o}cherer, T. Ibukiyama, M. Kaneko, F. Sato (eds.), World Scientific, New Jersey, 71-106, 2006
%**********************************************************

\bibitem{Gon1}
Goncharov, A. B.
\textit{The dihedral Lie algebras and Galois symmetries of $\pi_1^{(l)}(\mathbb{P}^1-(\{0,\infty\}\cup\mu_N))$},
Duke Math. J. 110, No. 3, 397-487, 2001
%**********************************************************


\bibitem{IKZ}
Ihara, K., Kaneko, M., and Zagier, D.
\textit{Derivation and double shuffle relations for multiple zeta values},
Compositio Math. 142, 307-338, 2006
%**********************************************************


\bibitem{Mac}
Machide, T.
\textit{Some restricted sum formulas for double zeta values},
Proc. Japan Acad. Ser. A 89, 51-54, 2013
%**********************************************************

\bibitem{OZ}
Ohno, Y., Zudilin, W.
\textit{Zeta stars},
Commun. Number Theory Phys., 2(2), 324-347, 2008
%**********************************************************

\bibitem{Rac}
Racinet, G.
\textit{Doubl\'{e}s melanges des polylogarithmes multiples aux racines de l'unit\'{e}},
Math. Inst. Hautes Etudes Sci. No. 95, 185-231, 2002
%**********************************************************

\bibitem{Zag1}
Zagier, D.
\textit{Evaluation of the multiple zeta values $\zeta(2,\ldots,2,3,2,\ldots,2)$},
Annals of Math. 175, 977-1000, 2012
%**********************************************************


\end{thebibliography}

\bibliographystyle{amsplain}

% Non-BibTeX users please use
%\begin{thebibliography}{}
%
% and use \bibitem to create references. Consult the Instructions
% for authors for reference list style.
%
%\bibitem{RefJ}
% Format for Journal Reference
%Author, Article title, Journal, Volume, page numbers (year)
% Format for books
%\bibitem{RefB}
%Author, Book title, page numbers. Publisher, place (year)
% etc
%\end{thebibliography}

\end{document}